\documentclass[11pt,leqno,a4paper]{amsart}
\usepackage{amssymb,amsmath,amscd,graphicx,fontenc,amsthm,mathrsfs, mathtools}
\usepackage[frame, cmtip, arrow, matrix, line, graph, curve]{xy}
\usepackage{graphpap, color}
\usepackage[mathscr]{eucal}
\usepackage{ifthen}

\usepackage{fancyhdr}
\usepackage{indentfirst}
\usepackage{paralist}
\usepackage{pstricks}
\usepackage{bbm}
\usepackage{tikz}
 
\newtheorem{theorem}{Theorem}[section]

\newtheorem{lemma}[theorem]{Lemma}
\newtheorem{proposition}[theorem]{Proposition}
\newtheorem{corollary}[theorem]{Corollary}

\theoremstyle{remark}

\numberwithin{equation}{section}

\newcommand{\NN}{{\mathbb{N}}}
\newcommand{\ZZ}{{\mathbb{Z}}}
\newcommand{\CC}{{\mathbb{C}}}

\newcommand{\HH}{{\mathbb{H}}}
\newcommand{\RR}{{\mathbb{R}}}

\newcommand{\QQ}{{\mathbb{Q}}}

\renewcommand{\SS}{{\mathbb{S}}}

\DeclareMathOperator{\rk}{rank}

\DeclareMathOperator{\vol}{vol}

\newcommand{\GL}{\mathrm{GL}}
\newcommand{\SL}{\mathrm{SL}}

\newcommand{\PSL}{\mathrm{PSL}}

\newcommand{\SU}{\mathrm{{SU}}}
\newcommand{\PSU}{\mathrm{{PSU}}}
\newcommand{\SO}{{\mathrm{SO}}}
\newcommand{\PSO}{{\mathrm{PSO}}}

\renewcommand{\t}{\mathsf{t}}

\newcommand{\ov}{\mathbf v}

\providecommand{\ve}{\mathbf{ e}}

\newcommand{\T}{{\mathsf{T}}}

\newcommand{\sg}{\mathsf{g}}

\DeclareMathOperator{\supp}{supp}
\providecommand{\vol}{\mathrm{vol}}

\newcommand{\diag}{\mathrm{diag}}

\newcommand{\q}{\mathsf q^{}_S}

\newcommand{\sq}{\mathsf q^0_S}

\newcommand{\SA}{\mathsf{A}}
\newcommand{\SK}{\mathsf{K}}

\newcommand{\sk}{\mathsf{k}}
\newcommand{\sa}{\mathsf{a}}

\newcommand{\SG}{\mathsf{G}}

\newcommand{\I}{\mathsf{I}}
\newcommand{\CT}{\mathsf{T}}
\newcommand{\CS}{\mathsf{S}}
\newcommand{\Ct}{\mathsf{t}}

\providecommand{\RR}{\mathbb{R}} \providecommand{\ZZ}{\mathbb{Z}}
\providecommand{\NN}{\mathbb{N}}

\newcommand{\Tree}{\mathcal{T}}
\newcommand{\Bding}{\mathcal{B}}

\newcommand{\hab}{\mathsf{s}\hspace{0.01in}}

\newcommand{\p}{\mathrm{p}}

\newcommand{\NT}{\mathbf N}
\newcommand{\VT}{\mathbf V}

\definecolor{cmd}{rgb}{1.0, 0.35, 0.21}

\begin{document}

\title[Quantitative Oppenheim for $S$-adic quadratic forms]{Quantitative Oppenheim conjecture for $S$-arithmetic quadratic forms of rank $3$ and $4$}

\author{Jiyoung Han}
\address{J.~Han. Research institute of Mathematics, Seoul National University
{\it Email address: jiyoung.han.math@snu.ac.kr}}


\date{}

\maketitle

\begin{abstract}
The celebrated result of Eskin, Margulis and Mozes (\cite{EMM}) and Dani and Margulis (\cite{DM}) on quantitative Oppenheim conjecture says that
for irrational quadratic forms $q$ of rank at least 5, 
the number of integral vectors $\ov$ such that $q(\ov)$ is in a given bounded interval
is asymptotically equal to the volume of the set of real vectors $\ov$ such that $q(\ov)$ is in the same interval.
 
In dimension $3$ or $4$, there are exceptional quadratic forms which fail to satisfy the quantitative Oppenheim conjecture. 
Even in those cases, one can say that almost all quadratic forms hold that two asymptotic limits are the same (\cite[Theorem 2.4]{EMM}).
In this paper, we extend this result to the $S$-arithmetic version. 
\end{abstract}


\section{Introduction}

\subsection*{History}
The Oppenheim conjecture, proved by Margulis \cite{Mar},
says that the image set $q(\mathbb Z^n)$ of integral vectors of an isotropic irrational quadratic form $q$ of rank at least 3 is dense in the real line (See \cite{Op} for the original statement of Oppenheim conjecutre).

Let $(a,b)$ be a bounded interval and let $\Omega\subset \RR^n$ be a convex set containing the origin. 
Dani and Margulis \cite{DM} and Eskin, Margulis and Mozes \cite{EMM} established a quantitative version of Oppenheim conjecture:
for a quadratic form $q$, let $\VT_{(a,b),\Omega}(T)$ be the volume of the region $T\Omega \cap q^{-1}(a,b)$ and $\NT_{(a,b),\Omega}(T)$ be the number of integral vectors in $T\Omega \cap q^{-1}(a,b)$.
They found that if an irrational isotropic quadratic form $q$ is of rank greater than or equal to 4 and is not a split form, then $\VT_{(a,b),\Omega}(T)$ approximates $\NT_{(a,b),\Omega}(T)$ as $T$ goes to infinity.
Moreover, in this case, the image set $q(\ZZ^n)$ is equidistributed in the real line.

Their works heavily rely on the dynamical properties of orbits of a certain semisimple Lie group, generated by its unipotent one-parameter subgroups, on the associated homogeneous space.

The $S$-arithmetic space is one of the optimal candidates to consider a generalization of their work,
since it has a lattice $\ZZ_S^n$ which is similar to the integral lattice $\ZZ^n$ in $\RR^n$. 
Borel and Prasad \cite{BP} extend Margulis' theorem to the $S$-arithmetic version and 
{Han, Lim and Mallahi-Karai} \cite{HLM} proved the quantitative version of $S$-arithmetic Oppenheim conjecture. 
See also \cite{GT} and \cite{KT} for flows on $S$-arithmetic symmetric spaces.
Ratner \cite{Ra} generalized her measure rigidity of unipotent subgroups in the real case to the cartesian product of real and $p$-adic spaces.

\subsection*{Main statements} Consider a finite set $S_f=\{p_1, \ldots, p_s\}$  of odd primes and let $S=\{\infty\}\cup S_f$. 
For each $p \in S$, denote by $\QQ_p$ the completion field of $\QQ$ with respect to the $p$-adic norm. If $p=\infty$, the norm $\|\cdot\|_\infty$ is the usual Euclidean norm and $\QQ_\infty=\RR$. Define $\QQ_S=\prod_{p\in S} \QQ_p$. 
The diagonal embedding of 
\[\begin{split}
\ZZ_S
&=\{mp_1^{n_1}\cdots p_s^{n_s} : m, n_1, \ldots, n_s \in \ZZ\}
\end{split}\]
is a uniform lattice subgroup in an additive group $\QQ_S$. 
We will define \emph{an $S$-lattice} $\Delta$ in $\QQ_S^n$ as a free $\ZZ_S$-module of rank $n$. 
An \emph{$S$-quadratic form} $\q$ is a collection of quadratic forms $q_p$ defined over $\QQ_p$, $p\in S$. 
We call $\q$ \emph{nondegenerate or isotropic} if every $q_p$ in $\q$ is nondegenerate or isotropic, respectively.
We say that $\q$ is \emph{rational} if there is a single rational quadratic form $q$ such that $\q=(\lambda_p q)_{p \in S}$,
where $\lambda_p \in \QQ_p-\{0\}$ and $\q$ is \emph{irrational} if it is not rational.

\vspace{0.08in}
\noindent \textbf{Notation.}
For a nonzero vector $\ov=(v_p)_{p\in S} \in \QQ_S^n$, the $p$-adic norm $\|\ov\|_p$ refers the $p$-adic norm $\|v_p\|_p$ of $p$-adic component $v_p \in \QQ_p^n$. We will use the following notation
\begin{equation*}\sigma(p)=\left\{\begin{array}{ll}
-1 & \text{if} \;p<\infty,\\
1 & \text{if} \;p=\infty.\\
\end{array}\right.
\end{equation*}

We use the notation $\ov/\|\ov\|_p^{\sigma}:=(v_p/|v_p|_p^\sigma)_{p\in S}$ for a vector consisting of unit vectors in each place $p\in S$.

For $p \in S$, we fix a star-shaped convex set $\Omega_p$ on $\QQ_p^n$ centered at the origin defined as
\[\Omega_p=\left\{ \ov \in \QQ_p^n : \| \ov\|_p < \rho_p(\ov/\| \ov\|^\sigma_p ) \right\},
\]
where $\rho_p$ is a positive function on the set of unit vectors in $\QQ_p^n$ and let $\Omega=\prod_{p\in S} \Omega_p$. If $p<\infty$, we will assume that $\rho_p$ is $(\ZZ_p-p\ZZ_p)$-invariant: for any $u \in \ZZ_p-p\ZZ_p$ and for any unit vector $v_p \in \ZZ_p^n-p\ZZ_p^n$,
\begin{equation*}\rho_p(u v_p)=\rho(v_p).  
\end{equation*} 

Let $I_p\subset \QQ_p$ be a bounded convex set of the form $(a,b)$ if $p=\infty$ and $a+p^b\ZZ_p$ if $p<\infty$. We call $I_p$ a \emph{$p$-adic interval}.
Define \emph{an $S$-interval} $\I_S=\prod_{p \in S} I_p$.
Denote by $\CT=(T_p)_{p\in S}\in \RR_{\ge 0}\times \prod_{p\in S_f} (p^{\ZZ}\cup\{0\})$ the collection of radius parameters $T_p$, $p \in S$.  
Consider $\CT\Omega=\prod_{p\in S}T_p\Omega_p$ the dilation of $\Omega$ by $\CT$.\\

We are interested in the number $\NT(\CT)=\NT(\q, \Omega, \I_S)(\CT)$ of vectors $\ov$ in $\ZZ_S^n \cap \CT\Omega$ such that $\q(\ov)\in \I_S$. 
Let $\mu$ be the product of Haar measures $\mu_p$ on $\QQ_p$, $p \in S$. We assume that $\mu_\infty$ is the usual Lebesgue measure on $\RR$ and $\mu_p(\ZZ_p)=1$ when $p<\infty$.
Let $\VT(\CT)=\VT(\q, \Omega, \I_S)(\CT)$ be the volume of the set $\{\ov \in \CT\Omega : \q(\ov)\in \I_S\}$ with respect to $\mu^n$ on $\QQ_S^n$.

Recall that a quadratic form of rank $4$ is \emph{split} if it is equivalent to $x_1x_4-x_2x_3$.
In \cite{HLM}, 
when an isotropic irrational quadratic form $\q$ is of rank $n\ge 4$ and does not contain a split form,
then as $\CT \rightarrow \infty$, $\NT(\CT)$ is approximated to $\VT(\CT)$. 
Here, we say that $\CT \rightarrow \infty$ if $T_p \rightarrow \infty$ for all $p\in S$.
Moreover, it is possible to estimate $\VT(\CT)$ in terms of $\I_S$ and $\CT$. As a result, there is a constant $C(\q, \Omega)>0$ such that 
$$\lim_{\CT\rightarrow \infty}\NT(\CT)= C(\q, \Omega) \mu(\I_S) |\CT|^{n-2},$$ 
where $|\T|=\prod_{p\in S} T_p$.

When $\q$ is of rank $3$ or $\q$ is of rank $4$ and contains a split form, there exist
quadratic forms such that $\NT(\CT)$ fails to approximate $\VT(\CT)$.
For instance, there is an irrational quadratic form $\q$ of rank $3$ and a sequence $(\CT_j)$ for a given $\varepsilon>0$ so that $\NT(\CT_j)\ge |\CT_j|(\log|\CT_j|)^{1-\varepsilon}$ (see \cite[Theorem 2.2]{EMM} and \cite[Lemma 9.2]{HLM}).

Even in the low dimensional cases, one can expect that for generic isotropic quadratic forms, $\NT(\CT)$ approximates $\VT(\CT)$,
which is our main theorem.

Here, the term \emph{generic} is with respect to the following measure:
fix some quadratic form $\sq$. One can identify $\SO(\sq)\setminus \SL_n(\QQ_S)$ with
the space of quadratic forms of the same discriminant with $\sq$.
Under this identification, one can assign a natural $\SL_n(\QQ_S)$-invariant measure on the space of quadratic forms.

\begin{theorem}\label{main thm} For almost all isotropic quadratic forms $\q$ of rank 3 or 4, as $\CT \rightarrow \infty$,
\[ \NT(\q, \Omega, \I_S)(\CT) \sim \lambda_{\q, \Omega}\; \mu(\I_S) |\CT|^{n-2},
\]
where $n=\rk(\q)$ and $\lambda_{\q,\Omega}$ is a constant depending on a quadratic form $\q$ and a convex set $\Omega$.
\end{theorem}

\textcolor{black}{For real quadratic forms of signature $(2,2)$, Eskin, Margulis and Mozes describe quadratic forms such that $\NT_{a,b,\Omega}(T)$ does not approximate to $\VT_{a,b,\Omega}(T)$ (\cite{EMM2}).}

On the generic quadratic forms and their Oppenheim conjecture-type problems, there is a work of Bourgain \cite{Bour} for real diagonal quadratic forms using analytic number theoretical method. 
Ghosh and Kelmer \cite{GK} showed another version of quantitative version of Oppenheim problem for real generic ternary quadratic forms and Ghosh, Gorodnik and Nevo \cite{GGN} extended this result for more general setting, such as for generic characteristic polynomial maps.
Recently, Athreya and Margulis \cite{AM} provided the bound of error terms $\NT_{(a,b),\Omega}(T)-\VT_{(a,b),\Omega}(T)$ for almost every real quadratic forms of rank at least 3.

For a pair of a quadratic form $q$ and a linear form $L$ of $\RR^n$, Gorodnik \cite{Gor} showed the density of $\{(q(\ov),L(\ov)) : \ov \in \mathcal P(\ZZ^n)\}\subseteq \RR^2$ under the certain assumption and Lazar \cite{Laz} generalized his result for the $S$-arithmetic setting. Sargent showed the density of $M(\{\ov \in \ZZ^n : q(\ov)=a\})\subseteq \RR$, where $q$ is a rational quadratic form and $M$ is a linear map (\cite{Sar1}). In \cite{Sar2}, he showed the quantitative version of his result.

In Section 2, we briefly review the symmetric space of the real and $p$-adic Lie groups.
In Section 3, we introduce an $S$-arithmetic alpha function, defined on $\SL_n(\QQ_S)/\SL_n(\ZZ_S)$ and equidistribution properties. 
We prove the main theorem in Section 4.

\subsection*{Acknowledgments} 
I would like to thank Seonhee Lim and Keivan Mallahi-Karai for suggesting this problem and providing valuable advices. 
This paper is supported by the Samsung Science and Technology Foundation under project No. SSTF-BA1601-03.

\section{Symmetric spaces}
Let us denote by $\SG$ an $S$-arithmetic group of the form $\SG=\prod_{p\in S} G_p$, where $G_p$ is a semisimple Lie group defined over $\QQ_p$, $p\in S$. 
An element of $\SG$ is $\sg=(g_p)_{p \in S}=(g_\infty, g_1, \ldots, g_s)$, where $g_i \in G_{p_i}$.
In most of cases, $\SG$ will be $\SL_n(\QQ_S):=\prod_{p \in S} \SL_n(\QQ_p)$, $n\ge 3$. 
Consider the lattice subgroup $\Gamma=\SL_n(\ZZ_S)$ of $\SG$ and the symmetric space $\SG/\Gamma$, which can be embedded in the space of unimodular $S$-lattices in $\QQ_S^n$.\\

The notation $\sq$ refers to a \emph{standard quadratic form}, which is the collection of quadratic forms such that
\[\begin{split}
&\left\{\begin{array}{l}
q_\infty(x_1,x_2,x_3)=2x_1x_3+x_2^2,\\
q_p(x_1,x_2,x_3)=2x_1x_3+a_1x_2^2,\;p\in S_f,\\
\end{array}\right. \hspace{0.865in}\text{if}\; \rk(\sq)=3,\\
&\left\{\begin{array}{l}
q_\infty(x_1,x_2,x_3,x_4)=2x_1x_4+a_1x_2^2+a_2x_3^2,\\ q_p(x_1,x_2,x_3,x_4)=2x_1x_4+a_1x_2^2+a_2x_3^2,\;p\in S_f,
\end{array}\right. \quad\text{if}\; \rk(\sq)=4,\\
\end{split}\] 
where $a_1,a_2 \in\{\pm1\}$ if $p=\infty$ and $a_1,a_2 \in\{\pm1,\pm u_0, \pm p, \pm pu_0\}$ if $p<\infty$. 
Here $u_0$ is some fixed square-free integer in $\QQ_p$ (See \cite[Section 4.2]{Se}).

In $\SO(\sq)$,  we fix a maximal compact subgroup $\SK$ of $\SO(\sq)$
\[\SK=K_\infty \times \prod_{p\in S_f} K_p=\left(\SO(q^0_\infty)\cap \SO(n)\right) \times \prod_{p\in S_f} 
\left(\SO(q^0_p)\cap \SL_n(\ZZ_p)\right)\]

and a diagonal subgroup 
\begin{equation}\label{diagonal group A}
\SA=\left\{\sa_{\Ct}=\prod_{p\in S} a_{t_{p_i}} : \Ct=(t_\infty, t_1, \ldots, t_s) \in \RR_{\ge0}\times p_1^{\ZZ}\times\cdots\times p_s^{\ZZ}\right\},\end{equation}
where
$a_{t_\infty}=\diag(e^{-t_\infty}, 1, \ldots, 1, e^{t_\infty})$ and $a_{t_p}=\diag(p^{t_p},1, \ldots, p^{-t_p})$, $p\in S_f$.
These groups $\SK$ and $\SA$ will be heavily used throughout the paper.  

We assume that $\CT$ and $\Ct$ have the relation
$$T_\infty=e^{t_\infty}\quad \text{and}\quad T_p=p^{t_p},\;\forall p\in S_f,$$
unless otherwise specified.
Also we briefly denote by $T_i$ or $t_i$ instead of $T_{p_i}$ or $t_{p_i}$ respectively.

In this section, we take $G_\infty=\SL_n(\RR)$ and $G_p=\GL_n(\QQ_p)$, $p<\infty$, where $n=3$ or $4$. Corresponding maximal compact subgroups are $\hat K_\infty=\SO(n)$ and $\hat K_p=\GL_n(\ZZ_p),$ respectively.
 
\subsection{The 3-dimensional hyperbolic space $\HH^3$}
Let $q^0_\infty$ be a standard isotropic quadratic form of rank $n$, $n=3,4$.
The special orthogonal subgroup $H_\infty=\SO(q^0_\infty)$ is one of the well-known linear groups $\SO(2,1)$, $\SO(3,1)$ and $\SO(2,2)$.
Let us examine the symmetric space of $H_\infty$ quotiented by $K_\infty$ and define a metric invariant by right multiplication.

\vspace{0.08in}
\noindent \textbf{Case i) $H_\infty=\SO(3,1)$:} Set $\HH^3$ to be the 3-dimensional hyperbolic space $\HH^3=\{z+ti : z=x+yj \in \CC\;\text{and}\; t \in \RR_{>0}\;\}$. 
The group $\SO(3,1)$ is locally isomorphic to $\PSL_2(\CC)$, 
which is the group $\mathrm{Isom}^+(\HH^3)$ of orientation-preserving isometries of $\HH^3$.
Since the stabilizer of the point $i$ in $\PSL_2(\CC)$ is the maximal compact subgroup $\PSU(2)$, 
we may identify the symmetric space $K_\infty\setminus \SO(3,1)$, 
which is isomorphic to $\PSU(2)\setminus \PSL_2(\CC)$, with the hyperbolic space $\HH^3$. 
Let $\p:\SO(3,1)\rightarrow K_\infty \setminus \SO(3,1)\simeq \HH^3$ be the projection given by $\p(g)=g.i.$
Define the metric $d_\infty$ of $K_\infty\setminus \SO(3,1)$ by 
\[d_\infty(g_1, g_2)=d_\infty(K_\infty g_1, K_\infty g_2):=d_{\HH^3}(\p(g_1), \p(g_2)).
\]

\vspace{0.08in}
\noindent \textbf{Case ii) $H_\infty=\SO(2,1)$:}
Consider the isomorphism $K_{\infty}\setminus \SO(2,1)\simeq\PSO(2)\setminus \PSL_2(\RR)\simeq \HH^2=\{ x+ti : x \in \RR\;\text{and}\; t \in \RR_{>0}\:\}$. We will use the notation $\p$ for the projection $\SO(2,1)\rightarrow \HH^2$ as well. We define the metric $d_\infty$ of $K_\infty\setminus \SO(2,1)$ by
$$d_\infty(g_1, g_2):=d_{\HH^2}(\p(g_1), \p(g_2)).$$

\vspace{0.08in}
\noindent \textbf{Case iii)$H_\infty=\SO(2,2)$:}
In the case of $H_\infty=\SO(2,2)$, consider the isomorphism between real vector spaces $\RR^4$ and $\mathcal M_{2}(\RR)$ given by
\[(x,y,z,w) \mapsto \left(\begin{array}{cc} x & y \\ z & w \end{array}\right).
\]

The split form $xw-yz$ of $\RR^4$ corresponds to the determinant of $\left(\begin{array}{cc} x & y \\ z & w \end{array}\right)$ in $\mathcal M_2(\RR)$.
Moreover, there is the local isomorphism from $\SL_2(\RR)\times \SL_2(\RR)$ to $\SO(2,2)$, which is induced by the action
\begin{equation}\label{local isom}
(g_1, g_2).\left(\begin{array}{cc} x & y \\ z & w \end{array}\right)=g^{}_1\left(\begin{array}{cc} x & y \\ z & w \end{array}\right)g^t_2.
\end{equation}

Hence we deduce that 
\[K_{\infty}\setminus \SO(2,2) \simeq_{loc} (\SO(2)\times \SO(2))\setminus (\SL_2(\RR)\times \SL_2(\RR)) \simeq \HH^2_1 \times \HH^2_2,
\]
where $\HH^2_1\simeq\HH^2\simeq \HH^2_2$. Note also that the action of $a_t$ splits into the action of $(b_t, b_t)$, where $b_t=\diag(e^{t/2}, e^{-t/2})$.
Put $$d_\infty=d_{\SO(2,2)}:=\max\{d_{\HH^2_1}, d_{\HH^2_2}\}.$$ 


\begin{lemma}\label{lemma 3.12 real} Let $H_\infty$ be one of $\SO(2,1)$, $\SO(3,1)$ or $\SO(2,2)$. 
Let $K_\infty$ be a maximal compact subgroup of $H_\infty$.  Then there exists a constant $C_\infty >0$ such that for any $t>0$ and for any $r \in (0, 2t)$,
\begin{equation}\label{eq lemma 3.12 real}
\left|\left\{k \in K_\infty : d_\infty(a^{}_t k a^{-1}_t, 1) \le r \right\}\right|<C_\infty e^{-2t+r},
\end{equation}
where $|\cdot|$ is the normalized Haar measure on $K_\infty$.
\end{lemma}

\begin{proof}

Assume that $H_\infty=\SO(3,1)\cong \SL_2(\CC)$ and $K_\infty\cong \SU(2)$.
For each $t >0$, $K_\infty$ acts transitively on the hyperbolic sphere $S_{2t}\subset \HH^3$ of radius $2t$ centered at $i$.

 Since $d_\infty$ is $H_\infty$-invariant,
\[\begin{split}
\left|\left\{k \in K_\infty : d_{\infty}(a^{}_t k a^{-1}_t, 1) \le r \right\}\right|
&=\left|\left\{k \in K_\infty : d_{\infty}(a^{}_t k , a^{}_t) \le r \right\}\right|\\
&=\left|\left\{y \in S_{2t} : d_{\HH^3}(y, e^{2t}i) \le r \right\}\right|.
\end{split}\]

The $K_\infty$-invariant measure of $S_{2t}$ is identified with the normalized Lebesgue measure of the unit sphere $\SS^2$ which is isomorphic to $\partial \HH^3$. 

Let $x=e^{2t}i$ and let $\theta$ be the angle between two geodesics from $i$ to $x$ and from $i$ to $y$ (Figure 1).
By the hyperbolic law of cosines
\[\cosh(d(x,y))=1+\sinh^2 (2t)\cdot (1-\cos\theta)=1+2\sinh^2(2t) \sin^2(\frac \theta 2),
\]
and since $d(x,y)\le r$,
we obtain that $\theta$ is bounded by $e^{-2t+r/2}$. Hence
\[\left|\left\{k \in K_\infty : d_{\infty}(a_t k a^{-1}_t, 1)\le r\right\}\right|\le C_\infty e^{2(-2t+r/2)}
\]
for some constant $C_\infty>0$.

The case of $H_\infty=\SO(2,1)\cong \SL_2(\RR)$ follows immediately from the compatibility between two embeddings $\SO(2,1)\hookrightarrow \SO(3,1)$ and $\HH^2 \hookrightarrow \HH^3$ with the projection $\p$ (see also \cite[Lemma 3.12]{EMM}).

If $H_\infty=\SO(2,2)$, using the local isometry \eqref{local isom}, $K_\infty\cong \SO(2)\times\SO(2)$ acts transitively on the product $S_{t}\times S_{t}$ of two hyperbolic spheres in $\HH^2\times \HH^2$. Thus we have
\[\begin{split}&\left|\left\{k \in K_\infty : d_\infty(a^{}_t k a^{-1}_t, 1) \le r \right\}\right|\\
&\hspace{0.3in}=\left|\{k_1\in \SO(2) : d_{\HH}(b_t k_1, b_t)\le r\}\times\{k_2\in \SO(2) : d_{\HH}(b_t k_2, b_t)\le r\}\right|\\
&\hspace{0.3in}< C \left(e^{-t+r/2}\right)^2=Ce^{-2t+r}.
\end{split}\]
\end{proof}
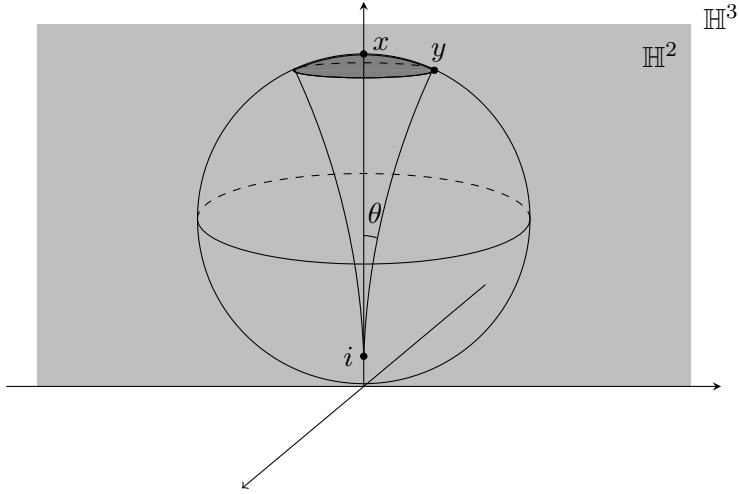
\begin{figure}[h!]
  \begin{center}
    \begin{tikzpicture}
      \fill[color=lightgray] (-4.3,4.8) rectangle (4.3,0);
      \filldraw[fill=gray, draw=black] (-0.9284,4.1919) arc (180:360:0.9284 and 0.1);
      \filldraw[fill=gray] (0.9284,4.1919) arc (65:115:2.1865 and 2.1865);
      \draw [->,>=stealth] (-4.7,0) -- (4.7,0);
      \draw [->,>=stealth] (0,0) -- (0,5.1);
      \draw [->,.=stealth] (1.6,1.35) -- (-1.6,-1.35);
      \node at (3.9,4.4) {$\HH^2$};
      \node at (4.7,4.9) {$\HH^3$};
      
      \fill[fill=black] (0,0.4) circle (1.4pt);
      \node at (-0.2,0.4) {$i$};
      
      \draw (0,2.2228) circle (2.1865);
      
      \draw (-2.1865,2.2228) arc (180:360:2.1865 and 0.6);
      \draw[dashed] (2.1865,2.2228) arc (0:180:2.1865 and 0.6);
      
      \fill[fill=black] (0,4.4093) circle(1.4pt);
      \node at (0.23, 4.55) {$x$};
      
      \fill[fill=black] (0.9284,4.1919) circle(1.4pt);  
      \node at (0.99, 4.4329) {$y$};

      \draw (0,0.4) arc (178:155:9.8497 and 9.8497);
      \draw (0,0.4) arc (2:180-155:9.8497 and 9.8497);
      
      \draw (-0.9284,4.1919) arc (180:360:0.9284 and 0.1);

      \draw[dashed] (0.9284,4.1919) arc (0:180:0.9284 and 0.1);
      
      \draw (0,2) arc (90:70:0.5 and 0.5);
      \node at (0.15,2.3) {$\theta$};

    \end{tikzpicture}
    \caption{The 3-dimensional hyperbolic space $\HH^3$. The measure of the set in \eqref{eq lemma 3.12 real} is equal to the Lebesgue measure of the grey area on the top of the sphere.}
  \end{center}
\end{figure}

\subsection{A tree in the building $\Bding_n$} For a prime $p$ and $n\ge 2$, let us briefly recall 
the Euclidean building $\Bding_n$, whose vertex set $\mathcal V(\Bding_n)$ is isomorphic to $\GL_n(\ZZ_p)\setminus \GL_n(\QQ_p)$. Recall the Cartan decomposition (\cite[Theorem 3.14]{PR}) 
$$\GL_n(\QQ_p)=\GL_n(\ZZ_p)\cdot\hat A\cdot \GL_n(\ZZ_p)=\GL_n(\ZZ_p)\cdot \hat A^+ \cdot\GL_n(\ZZ_p),$$ where
\[\begin{split}
\hat A&=\{\diag(p^{m_1}, p^{m_2}, \ldots, p^{m_n}) : m_1, m_2, \ldots, m_n \in \ZZ\},\\
\hat A^+&=\{\diag(p^{m_1}, p^{m_2}, \ldots, p^{m_n}) : 0\le m_1\le m_2\le\cdots\le m_n\}\subseteq \hat A.
\end{split}\]

Consider the space of free $\ZZ_p$-modules $L$ of rank $n$ in $\QQ_p^n$
on which $\GL_n(\QQ_p)$ acts by
\[g.(\ZZ_p \ov_1\oplus\cdots\oplus \ZZ_p\ov_n)=\ZZ_p(\ov_1g)\oplus \cdots \oplus \ZZ_p(\ov_ng), 
\]
where $g\in \GL_n(\QQ_p)$ and $\ov_1,\ldots,\ov_n \in \QQ_p^n$ are $\ZZ_p$-basis of $L$.
Two rank-$n$ free $\ZZ_p$-modules $L_1$ and $L_2$ are said to be \emph{equivalent} if $L_1=p^m L_2$ for some $m \in \ZZ$. 

The building $\Bding_n$ is an $(n-1)$-dimensional CW complex whose vertices are
equivalence classes of free $\ZZ_p$-modules $[L]$ of rank $n$.
Vertices $[L_0]$, $\ldots$, $[L_k]$ in $\Bding_n$ form a $k$-simplex
if there are representatives $L_0$, $\ldots$, $L_k$ in $\QQ_p^n$ such that
$$p L_0 \subsetneq L_1 \subsetneq \cdots \subsetneq L_k \subsetneq L_0.$$   

In particular, at each vertex $[L]$ in $\Bding_n$, the number of adjacent vertices with $[L]$ equals the number of proper nontrivial subspaces of the space $(\ZZ_p/p\ZZ_p)^n$.
Note that at most $(n-1)$ vertices form a simplex in $\Bding_n$:
there is an $\ZZ_p$-generating set $\{\ov_1, \ov_2, \ldots, \ov_n\}$ of $L$ such that
\[\begin{split}
&[pL]=[\langle p\ov_1, p\ov_2, \ldots, p\ov_n\rangle] \subsetneq [\langle\ov_1, p\ov_2, \ldots, p\ov_n\rangle] \subsetneq\\
&\hspace{0.2in} [\langle\ov_1, \ov_2, p\ov_3, \ldots, p\ov_n\rangle]\subsetneq [\langle\ov_1, \ldots, \ov_{n-1}, p\ov_n\rangle \subsetneq [\langle\ov_1, \ldots, \ov_n\rangle]=[L].
\end{split}\]

Vertices $[L]$ in $\Bding_n$ is also denoted by $n$ by $n$ matrices whose rows are $\ZZ_p$ generators of $L$. 
By right multiplication of $\GL_n(\ZZ_p)$,
they are represented by upper triangle matrices whose diagonal entries are in $p^{\ZZ_{\ge0}}$.

In $\Bding_n$, an apartment $\mathcal A$ is defined as follows. The vertex set $\mathcal V(\mathcal A)$ is 
\[\mathcal V(\mathcal A)=\{[ag]\in \GL_n(\ZZ_p)\setminus \GL_n(\QQ_p) : a \in \hat A\}
\]
for some $g\in\GL_n(\QQ_p)$. A $k$-simplicial complex is in $\mathcal A$ if its vertices are all contained in $\mathcal A$.
Denote by $\mathcal A_0$ the apartment whose vertex set is
\[
\{[\diag(p^{m_1}, \ldots, p^{m_n})] : m_i \in \ZZ_{\ge 0}\}.
\]
Note that $\mathcal A$ is isometric to $\RR^{n-1}$ and
$\mathcal V(\mathcal A_0)$ is a lattice in $\RR^{n-1}$ (See Figure 2),
which is the reason that $\Bding_n$ is called a Euclidean building.
There is the natural covering map $\p:\Bding_n\rightarrow \mathcal A_0$ given by
\[\p : \left[\left(\begin{array}{cccc}
p^{m_1} & v_{12} & \cdots & v_{1n}\\
& p^{m_2} & \cdots & v_{2n}\\
& & \ddots & \vdots\\
& & & p^{m_n}
\end{array}\right)\right]\mapsto [\diag(p^{m_1}, p^{m_2}, \ldots, p^{m_n})],
\]
where the matrix in the above map is a reduced type, i.e., $m_i \ge 0$ and $\gcd\{p^{m_j}, v_{ij}\}=1$, $1\le i<j\le n$.
  
We give a distance $d$ on $\mathcal V(\Bding_n)$ by $d([L_1],[L_2])=\ell$, where $\ell$ is the minimal number of edges connecting $[L_1]$ and $[L_2]$.  

Let $a:\ZZ \rightarrow \mathcal V(\mathcal A_0)$ be any geodesic in the vertex set of $\mathcal A_0$. Then the inverse image
\[\p^{-1}(\{a(m) : m \in \ZZ\})
\]
is a tree in $\Bding_n$. 
If the above set $\{a(m)\}$ is generated by one element, then $\p^{-1}(\{a(m)\})$ is a regular tree. 
More basic properties about buildings can be found in \cite{AB}, \cite{Robertson} for Euclidean buildings, and see also \cite{Setree} for Bruhat-Tits trees. 

\begin{figure}[h!]
  \begin{center}
    \begin{tikzpicture}
		\path[draw, gray] (-4*1.25,0) -- (4*1.25,0);  
		\path[draw, gray] (-4*1.25, 3^.5/2*1.25) -- (4*1.25, 3^.5/2*1.25);
		\path[draw, gray] (-4*1.25, 3^.5*1.25) -- (4*1.25, 3^.5*1.25);
		\path[draw, gray] (-4*1.25, -3^.5/2*1.25) -- (4*1.25, -3^.5/2*1.25);
		\path[draw, gray] (-4*1.25, -3^.5*1.25) -- (4*1.25, -3^.5*1.25);
		\path[draw, gray] (-4*1.25, 27^.5/2*1.25) -- (4*1.25, 27^.5/2*1.25);
		
		\path[draw, gray] (-1*1.25, -3^.5*1.25) -- (1.5*1.25, 27^.5/2*1.25);
		\path[draw, gray] (0*1.25, -3^.5*1.25) -- (2.5*1.25, 27^.5/2*1.25);
		\path[draw, gray] (1*1.25, -3^.5*1.25) -- (3.5*1.25, 27^.5/2*1.25);
		\path[draw, gray] (2*1.25, -3^.5*1.25) -- (4*1.25, 3^.5*1.25);
		\path[draw, gray] (-2*1.25, -3^.5*1.25) -- (0.5*1.25, 27^.5/2*1.25);
		\path[draw, gray] (-3*1.25, -3^.5*1.25) -- (-0.5*1.25, 27^.5/2*1.25);
		\path[draw, gray] (-4*1.25, -3^.5*1.25) -- (-1.5*1.25, 27^.5/2*1.25);
		
		\path[draw, gray] (-1.5*1.25, 27^.5/2*1.25) -- (1*1.25, -3^.5*1.25);
		\path[draw, gray] (-0.5*1.25, 27^.5/2*1.25) -- (2*1.25, -3^.5*1.25);
		\path[draw, gray] (0.5*1.25, 27^.5/2*1.25) -- (3*1.25, -3^.5*1.25);
		\path[draw, gray] (1.5*1.25, 27^.5/2*1.25) -- (4*1.25, -3^.5*1.25);
		\path[draw, gray] (-2.5*1.25, 27^.5/2*1.25) -- (0*1.25, -3^.5*1.25);
		\path[draw, gray] (-3.5*1.25, 27^.5/2*1.25) -- (-1*1.25, -3^.5*1.25);
		\path[draw, gray] (-4*1.25, 3^.5*1.25) -- (-2*1.25, -3^.5*1.25);

		\path[draw, blue] (-4*1.25,-3^.5*1.25) -- (-3*1.25, -3^.5*1.25);
		\path[draw, blue] (-3*1.25,-3^.5*1.25) -- (-2.5*1.25, -3^.5/2*1.25);
		\path[draw, blue] (-2.5*1.25,-3^.5/2*1.25) -- (-1.5*1.25, -3^.5/2*1.25);
		\path[draw, blue] (-1.5*1.25, -3^.5/2*1.25) -- (-1*1.25, 0);
		\path[draw, blue] (-1*1.25,0) -- (0,0);
		\path[draw, blue] (0,0) -- (0.5*1.25, 3^.5/2*1.25);
		\path[draw, blue] (0.5*1.25, 3^.5/2*1.25) -- (1.5*1.25, 3^.5/2*1.25);
		\path[draw, blue] (1.5*1.25, 3^.5/2*1.25) -- (2*1.25, 3^.5*1.25);
		\path[draw, blue] (2*1.25, 3^.5*1.25) -- (3*1.25, 3^.5*1.25);
		\path[draw, blue] (3*1.25, 3^.5*1.25) -- (3.5*1.25, 27^.5/2*1.25);
		\path[draw, blue] (3.5*1.25, 27^.5/2*1.25)--(4*1.25, 27^.5/2*1.25);
		\fill[fill=blue] (0.5*1.25, 3^.5/2*1.25) circle (1.5pt);
		\fill[fill=blue] (2*1.25, 3^.5*1.25) circle (1.5pt);
		\fill[fill=blue] (3.5*1.25, 27^.5/2*1.25) circle (1.5pt);
		\fill[fill=blue] (-2.5*1.25, -3^.5/2*1.25) circle (1.5pt);
		\fill[fill=blue] (-4*1.25, -3^.5*1.25) circle (1.5pt);
		
		\fill[fill=black] (-1*1.25,0) circle (1.4pt);
		\node at (-1*1.25, -0.3*1.25) {\scriptsize$\mathrm{Id}_3$};
		\draw[-, line width=1.5pt] (-1*1.25,0) -- (0,0);
		\node at (0.7*1.25, -0.3*1.25) {\scriptsize$\diag(1,1,p)$};
		\node at (-0.5*1.25, 3^.5/2*1.25+0.2*1.25) {\scriptsize$\diag(1,p,p)$};
		\draw[->, line width=1.5pt] (0,0) -- (0.5*1.25, 3^.5/2*1.25);
		
		\node at (0.5*1.25+1.75, 3^.5/2*1.25+0.3) {\scriptsize$\diag(1,p,p^2)\simeq \diag(p^{-1},1,p)$};
				     
    \end{tikzpicture}
    \caption{Apartment $\mathcal A_0$ of $\Bding_3$. $K_p\setminus\SO(2x_1x_3-x_2^2)$ is embedded in the inverse image of the blue line.}
  \end{center}
\end{figure}

\begin{lemma}\label{lemma 3.12 p-adic} Let $H_p$ be a $\QQ_p$-rank one connected and simply connected semisimple algebraic subgroup of $\GL_n(\QQ_p)$.
Denote the Cartan decomposition of $H_p$ by $H_p=K_p\cdot \hat A\cdot K_p
$, where $K_p=H_p\cap \GL_n(\ZZ_p)$. 
Then the symmetric space $K_p \setminus H_p$ is a tree.
\end{lemma}
\begin{proof} 
Note that there is the natural embedding from $K_p\setminus H_p$ to $\GL_n(\ZZ_p)\setminus \GL_n(\QQ_p)$ given by $K_p g \mapsto \GL_n(\ZZ_p)g$, $g\in H_p$. For $g \in H_p$, we denote $\GL_n(\ZZ_p)g$ by $K_p g$ and $\{\GL_n(\ZZ_p)g : g \in H_p\}$ by $K_p\setminus H_p$.

Since $H_p$ is of $\QQ_p$-rank one, up to an appropriate conjugation, $K_p\hat A$ is generated by some element $a=\diag(p^{m_1},\ldots,p^{m_n})$ with $m_1\le\cdots\le m_n$ and $m_1\neq 0$. Hence we may assume that $K_p \hat A=\langle a\rangle$.
Choose a geodesic $\{a(m) : m\in \ZZ\}$ in $\mathcal V(\mathcal A_0)$ containing $K_p \hat A$.
Then $K_p\setminus H_p$ is a subset of a tree $\p^{-1}(\{a(m) : m\in \ZZ\})$. 
\end{proof}

Let us show that for a $4$-dimensional nondegenerate isotropic quadratic form $q_p$ on $\QQ^4_p$, if it is not a split form, then the special orthogonal group $\SO(q_p)$ is of $\QQ_p$-rank one, similar to the case of real. 
\begin{lemma}\label{p-adic (3,1)}
Let $q_p$ be an isotropic non-split quadratic form of rank $4$. If we denote $\SO(q_p)=K\cdot \hat A\cdot K$, where $K=\SO(q_p)\cap \SL_4(\ZZ_p)$, then $\hat A$ is isomorphic to the diagonal subgroup $\left\{\diag(p^{-t},1,1, p^{t}) :  t \in \ZZ_{\ge 0}\right\}$. 
\end{lemma}
\begin{proof} Without loss of generality, let $q_p(x,y,z,w)=2xw+\alpha_1y^2+\alpha_2z^2$, where $\alpha_1, \alpha_2\in \{1, u_0, p, u_0p\}$ for some fixed $1\le u_0 \le p-1$, square-free over $\QQ_p$. 
Since $q_p$ is not a split form, $\alpha_2/\alpha_1\neq -1$.
Hence, up to a change of variables, we may further assume that 
\begin{enumerate}
\item $\alpha_2/\alpha_1 \in \{p, pu^{}_0, pu^{-1}_0\}$ if $-1$ is square-free;
\item $\alpha_2/\alpha_1 \in \{p, pu^{}_0, pu^{-1}_0, u^{}_0\}$ otherwise.
\end{enumerate}

Since any semisimple element of a connected algebraic group is contained in a maximal abelian subgroup and two maximal abelian subgroups are conjugate to each other (\cite[Theorem 6.3.5]{Sp}),
we may assume that
$$\left\{\diag(p^{-t}, 1, 1, p^t) : t \in \ZZ_{\ge0}\right\} \subseteq \hat A.$$

If the $\QQ_p$-rank of $\hat A$ is larger than one, there is an element $a=(a_{ij})_{1\le i,j\le 4}\in \hat A-\SL_4(\ZZ_p)$
such that $a\neq\diag(p^{-t},1,1,p^t)$ for any $t$. 

From the commutativity, $\diag(p^{-t},1,1,p^t)\hspace{0.025in}a=a\hspace{0.025in}\diag(p^{-t},1,1,p^t)$ for any $t\in \ZZ$, it follows that
\[a=\hspace{-0.05in}\left(\begin{array}{cccc}
a_{11} & 0 & 0 & 0 \\
0 & a_{22} & a_{23} & 0 \\
0 & a_{32} & a_{33} & 0 \\ 
0 & 0 & 0 & a_{44} \\
\end{array}\right)\hspace{-0.05in}=\hspace{-0.05in}\left(\begin{array}{cccc}
a_{11} & 0 & 0 & 0 \\
0 & 1 & 0 & 0 \\
0 & 0 & 1 & 0 \\ 
0 & 0 & 0 & a_{44} \\
\end{array}\right)\hspace{-0.05in}\left(\begin{array}{cccc}
1 & 0 & 0 & 0 \\
0 & a_{22} & a_{23} & 0 \\
0 & a_{32} & a_{33} & 0 \\ 
0 & 0 & 0 & 1 \\
\end{array}\right).
\]

Since $a\in \SO(q_p)$, it satisfies that for any $(x,y,z,w)\in \QQ_p^4$, 
\[
\begin{split}
&2xw+\alpha_1 y^2+\alpha_2 z^2\\
&\hspace{0.2in}= 2(a_{11}x)(a_{44}w)+\alpha_1(a_{22}y+a_{23}z)^2+\alpha_2(a_{32}y+a_{33}z)^2\\
&\hspace{0.2in}= 2\left(a_{11}a_{44}\right)xw+\left(\alpha_1 a^{2}_{22}+\alpha_2 a^2_{32}\right)y^2+2\left(\alpha_1a_{22}a_{23}+\alpha_2a_{32}a_{33}\right)yz\\
&\hspace{0.2in}\quad+\left(\alpha_1 a^{2}_{23}+\alpha_2a^2_{33}\right)z^2.
\end{split}
\]

Since $a_{11}$, $a_{44}$ have no restriction
except $a_{11}a_{44}=1$, by multiplying appropriate $(p^{-t},1,1,p^t)$ and $(u,1,1,u^{-1})$ to $a$, where $t\in \ZZ$ and $u \in \ZZ_p-p\ZZ_p$, we may assume that $a_{11}=a_{44}=1$.
Hence we need to find $(a_{22},a_{23},a_{32},a_{33})\notin \ZZ_p^4$ such that (a) $a^{2}_{22}+(\alpha_2/\alpha_1) a^2_{32}=1$, (b) $a_{22}a_{23}+(\alpha_2/\alpha_1)a_{32}a_{33}=0$, (c) $(\alpha_1/\alpha_2)a^{2}_{23}+a^2_{33}=1$.\\

If $\alpha_2/\alpha_1=p$, $pu^{}_0$, $pu^{-1}_0$,
let us denote $\alpha_2/\alpha_1=pu$, where $u\in \{1, u^{}_0, u^{-1}_0\}$.

Suppose that $(a_{22}, a_{32})\notin \ZZ_p^2$ and let $a^2_{22}+pua^2_{32}=\sum_{i\ge m} c_ip^i$, where $m\in \ZZ$ is the first term such that $c_m\neq 0$. Denote by $\nu:\QQ_p\rightarrow \ZZ$ the $p$-adic valuation. Then
\begin{enumerate}
\item If $|a_{22}|_p\ge|a_{32}|_p$, then $m$ is even and $\nu(a_{22})<0$. Then $c_m=\left(a_{22} p^{m/2}\right)^2 \mod p^m$ and $m=2\nu(a_{22})<0$ which contradicts to (a).
\item If $|a_{22}|_p <|a_{32}|_p$, then $m$ is odd and $\nu(a_{32})<0$. Then $c_m=u\left(a_32 p^{(m-1)/2}\right)^2 \mod p^m$ and $m=2\nu(a_{32})+1<0$ which is also impossible according to (a).
\end{enumerate}

Hence $(a_{22},a_{32})\in \ZZ_p^2$. Similarly, $(a_{23},a_{33})\in \ZZ_p^2$.

Now, assume $\alpha_2/\alpha_1=u_0$. In this case, $-1$ is a square.
If there exists $(a_{22},a_{23},a_{32},a_{33})$ satisfying (a) through (c), then $u_0=-(a_{22}/a_{32})^2$ 
which contradict to the fact that $u_0$ is square-free.
\end{proof}

If $q_p$ is an isotropic quadratic form of rank 3, 
$\SO(q_p)$ is isomorphic to $\SL_2(\QQ_p)$. 
Then the symmetric space $K_p\setminus \SO(q_p)$ is isomorphic to a subtree of the $(p+1)$ regular tree $\Tree_p$.
By Lemma \ref{lemma 3.12 p-adic} and Lemma \ref{p-adic (3,1)}, 
if $q_p$ is an isotropic non-split quadratic form of rank 4, 
then $K_p\setminus\SO(q_p)$ is a subtree of the building $\Bding_n$.
Lastly, if $q_p$ is a split form of rank 4, then $\SO(q_p)$ is locally isomorphic to $\SL_2(\QQ_p)\times\SL_2(\QQ_p)$
as in \eqref{local isom}.  
Then $K_p\setminus \SO(q_p)$ is properly embedded to the product $\Tree_p \times \Tree_p$ of two $(p+1)$-regular trees.

\begin{corollary}Let $q^0_p$ be a standard $p$-adic isotropic quadratic form of rank $n=3$ or $4$. 
Let $K_p=\SO(q_p) \cap \SL_n(\ZZ_p)$. For any $r\in (0,2t)$,
\begin{equation}\label{p-adic counting}
\left|\left\{k\in K_p : d_p(a^{}_tka^{-1}_t, 1) \le r \right\}\right| \le \frac p {p+1} p^{-(2t-r)}.
\end{equation}
\end{corollary}
\begin{proof} 
Suppose that $q_p$ is an isotropic quadratic form of rank 3 or of rank 4 and non-split. 
By Lemma \ref{lemma 3.12 p-adic} and Lemma \ref{p-adic (3,1)}, $K_p\setminus\SO(q_p)$ is isomorphic to a tree.
Since $d_p$ is right $K_p$-invariant, the left hand side of \eqref{p-adic counting} is
\[\left|\left\{k\in K_p : d_p(a^{}_tk, a^{}_t) \le r \right\}\right|
=\frac{\#\{K_pa_tk : d_{tree}(K_pa_tk, K_pa_t)\le r\}}{\#\{K_pa_tk : k\in K_p\}}.
\]

For any $k\in K_p$, consider the geodesic segment $[K_p, K_pa_tk]$ in $\Tree_p$.
If $d_p(a_tk, a_t)\le r$, since $K_p \setminus \SO(q_p)$ is a tree,
$[K_p, K_pa_tk]$ and $[K_p, K_pa_t]$ have the common segment of length at least $2t-\lceil r/2 \rceil$.
Therefore
\[\begin{split}
\left|\left\{k\in K_p : d_p(a^{}_tk, a^{}_t) \le r \right\}\right|
&=\frac {\#\{K_pa_tk : K_pa_{t-\lceil r/4 \rceil}k=K_pa_{t-\lceil r/4 \rceil} \}} {\#\{K_pa_tk : k \in K_p\}}\\
& =\frac{1}{\#\{K_pa_{t-\lceil r/4 \rceil}k : k \in K_p\}},
\end{split}\]
where $\lceil r\rceil$ is the largest integer less than or equal to $r$.

Since $\SO(2xz-y^2)(\QQ_p)$ is embedded in $\SO(q_p)$, for any $t\in \NN$, there are at least $(p+1)p(p^2)^{(t-1)-\lceil r/4 \rceil}$ vertices in the sphere of radius $2t-\lceil r/2 \rceil$ in $K_p \setminus \SO(q_p)$.

In the split case, the action of $a_{2t}=\diag(p^{2t},1,1,p^{-2t})$ on $K_p\setminus \SO(q_p)$ is converted to the action of $(b_t,b_t)$ on $\SL_2(\ZZ_p)\setminus\SL_2(\QQ_p)\times\SL_2(\ZZ_p)\setminus\SL_2(\QQ_p)\cong \Tree_p\times\Tree_p$, where $b_t=\diag(p^t,p^{-t})$. Then similar to the case of $\SO(2,2)(\RR)$, we obtain the inequality \eqref{p-adic counting}.
\end{proof}
\section{Equidistribution property for generic points}
 For $x \in \SG/\Gamma$, let $\Delta_x$ be an $S$-lattice in $\QQ_S^n$ associated to $x$, i.e., if $x=\sg\Gamma$, $\Delta_x=\sg\ZZ^n_S$. 
We call $L\subseteq \QQ^n_S$ \emph{a $\Delta_x$-rational subspace} if $L$ is generated by elements in $\Delta_x$. 
If $L$ is a $\Delta_x$-rational subspace, $\Delta_x \cap L$ is an $S$-lattice in $L$. 
Let $d(L)=d_{\Delta_x}(L)$ be a covolume of $\Delta_x \cap L$ in $L$. 
If $\{\ov_1, \ldots, \ov_j\}$ is an $\ZZ_S$-basis of a $j$-dimensional $\Delta_x$-rational subspace $L$, 
then $d(L)$ is given by
\[d(L)=\prod_{p \in S} \|\ov_1\wedge \cdots \wedge \ov_j\|_p.
\]
Here for each $p \in S$, the $p$-adic norm $\|\cdot\|_p$ of $\QQ_p$ is canonically extended to $\bigwedge^j(\QQ_p^n)$,
{which we also denote by $\|\cdot\|_p$.}

Define a function $\alpha^{}_S : \SG/\Gamma \rightarrow \RR_{>0}$ by
\[\alpha^{}_S(x)=\sup\left\{\frac 1 {d(L)} : L\;\text{is}\; \Delta_x\text{-rational subspace}\;\right\}.
\]

\textcolor{black}{The following lemma is originated from \cite{Sch}, stated for lattices in $\RR^n$. The statement and the proof of the $S$-arithmetic version can be found in \cite{HLM}.}
\begin{lemma}\label{Schmidt lemma} For a bounded function $f : \QQ_S^n\rightarrow \RR$ of compact support, the \emph{Siegel transform} $\tilde f$ of $f$ is defined by 
\[
\tilde f (x)=\sum_{\ov \in \Delta_x} f(\ov),\quad x \in \SG/\Gamma.\]

Then there is a constant $C_f>0$ such that $\tilde f(x)<C_f\alpha^{}_S(\Delta_x)$ for any $x \in \SG/\Gamma$.
\end{lemma}

By Lemma \ref{Schmidt lemma} and Proposition \ref{lemma 3.10}, 
one can easily deduce that a Siegel transform $\tilde f$ is in $\mathcal L^s$.

\begin{proposition}\cite[Lemma 3.9]{HLM}\label{lemma 3.10} The function $\alpha^{}_S$ on $\SL_n(\QQ_S)/\SL_n(\ZZ_S)$ belongs to every $\mathcal L^s$, $1\le s < n$.
\end{proposition}

In this section, we prove the equidistribution theorem for functions bounded by $\alpha^{}_S$.
We first recall the Howe-Moore theorem for $\SL_n(\QQ_S)$.
The statement for more general settings can be found in \cite{Benoist} or \cite{Oh}.

\begin{theorem}\label{matrix coeff}  Consider $\SG=\SL_n(\QQ_S)$ and $\Gamma=\SL_n(\ZZ_S)$, $n\ge 3$. 
Take a maximal subgroup $\hat \SK$ of $\SG$ as $\hat \SK=\underset{p\in S}{\prod}{\hat K_p} =\SO(n) \times \underset {p \in S_f}{\prod} \SL_n(\ZZ_p)$.  
Then, there exist constants $\lambda_p>0$, $p \in S$ satisfying the following:
if $f_1, f_2 \in \mathcal L^2_0(\SG/\Gamma)$ satisfy either
\begin{enumerate}[(1)]
\item both $f_1$ and $f_2$ are smooth of compact support or
\item both $f_1$ and $f_2$ are left $\hat \SK$-invariant,
\end{enumerate}
then there exists a constant $C_{f_1, f_2}>0$ so that
\[\int_{\SG/\Gamma} f_1(\sg x)f_2(x) dx \le C_{f_1,f_2}  e^{-\lambda_\infty d(g_\infty, 1)} \prod_{p \in S_f} p^{-\lambda_p d(g_p, 1)},
\]
where $\sg=(g_p)_{p\in S} \in \SG$.
\end{theorem}

Recall that $\sa^{}_t$ is a diagonal element of $\SA$ defined in \eqref{diagonal group A}.

\begin{lemma}\label{lemma 3.12} Let $\q$ be a nondegenerate isotropic quadratic form of rank $n=3$ or $4$.
Let $\SK$ be a maximal compact subgroup of $\SO(\q)$ as in Section 3.
Then there exist $\Ct_0\succ 0$ and a constant $C_S>0$ such that for $\Ct \succ \Ct_0$ and for any $r>0$,
\begin{equation}\label{eq lemma 3.12}\left|\left\{\sk \in \SK : \sum_{p \in S} d_p(\sa^{}_\Ct \sk \sa^{-1}_\Ct,1) \le r \right\}\right|< 
C_S\; r^{|S_f|}\exp {\left( r -\sum_{p\in S} t_p\right)}.
\end{equation}
\end{lemma}

\begin{proof} Let $s=|S_f|$ and consider $(k_1, \ldots, k_s) \in \ZZ_{\ge0}$ such that $\sum_i k_i \le r$. From Lemma~\ref{lemma 3.12 real} and Lemma~\ref{lemma 3.12 p-adic}, 
\[
\begin{split}
&\left|\left\{\sk=(k_p) \in \SK : \begin{array}{c}
d_\infty (a^{}_{t_\infty} k_\infty a^{-1}_{t_\infty}, 1)\le r-\sum_{i=1}^s k_i\;\text{and}\\
d_{p_i}(a^{}_{t_{p_i}} k_{p_i} a^{-1}_{t_{p_i}},1)\le k_i, 1\le i\le s\end{array}\right\}
\right|\\
&\hspace{1.7in}< C_\infty\;e^{(r-\sum_i k_i)-t_\infty} \times \prod_{i=1}^s \frac {p_i}{p_i+1}\; {p_i}^{{k_i}-2t_{p_i}}.
\end{split}
\] 

The number of such nonnegative vectors $(k_1, \ldots, k_s)$ is bounded above by $r^s$.
Since $e<p$ for any odd prime $p$, the inequality \eqref{eq lemma 3.12} holds for $C_S:=C_\infty \prod_i p_i/(p_i+1)$.
\end{proof}

We say that a sequence $(\Ct_j)$ is \emph{divergent} if $(\Ct_j)$ escapes any bounded subset of $\RR_{>0}\times \prod_{p\in S_f}p^{\ZZ}$ as $j\rightarrow \infty$.
Recall that in Theorem \ref{main thm}, we are interested in the asymptotic limit of the given quantity when every component $t_p$ of $\Ct$ goes to infinity.
However, for the proof of the main theorem, we need to consider more general divergence in the next proposition and corollary.

\begin{proposition}\label{prop 3.13} Let $\q$ and $\SK$ be as in Lemma \ref{lemma 3.12} and let $\phi$ be a continuous function on $\SG/\Gamma=\SL_n(\QQ_S)/\SL_n(\ZZ_S)$. 
Assume that there are constants $s\in (0,n/2)$ and $C_{\phi,s}>0$ for which
\[|\phi(x) |< C_{\phi,s} \alpha(x)^s_S,\; \forall x \in \SG/\Gamma.
\]

 Then for any nonnegative bounded function $\nu$ on $\SK$
 and a divergent sequence $(\Ct_j)$, 
\begin{equation}\label{eq prop 3.13} \lim_{j \rightarrow \infty} \int_\SK \phi(\sa_{\Ct_j} \sk x_0)\nu(\sk) dm(\sk)
= \int_{\SG/\Gamma} \phi\; d\sg \int_\SK \nu\; dm
\end{equation}
for almost all $x_0 \in \SL_n(\QQ_S)/\SL_n(\ZZ_S)$.
\end{proposition}
\begin{proof}
 Let $\varepsilon >0$ be arbitrary. For any function $f$ on $\SG/\Gamma$, define
 \[ A_{\Ct}f(x)=\int_{\SK} f(\sa_{\Ct} \sk x)\nu(\sk) dm(\sk).
 \]
 
 For $\Ct\succ 0$, let us denote by $\hab(\Ct)=\sum_{p\in S} t_p$.
 We claim that one can find $\Ct_0\succ 0$ and constants $C,\;\lambda>0$ such that for any $\Ct\succ \Ct_0$, 
 \begin{equation}\label{eq prop 3.13 (1)}\left|E_{\Ct}:= \left\{
 x \in \SG/\Gamma : 
\left|A_{{\Ct}'}\phi(x) - \int\phi \int \nu \right| > \varepsilon
\;\text{for some}\;\Ct'\succ \Ct\;\right\}\right| < Ce^{-\lambda \hab(\Ct)}.
\end{equation}

Then \eqref{eq prop 3.13 (1)} implies that
\[\left|\left\{x\in \SG/\Gamma :\left| \begin{array}{l}
\lim_{j \rightarrow \infty} \int_\SK \phi(\sa_{\Ct_j} \sk x_0)\nu(\sk) dm(\sk)\\
\hspace{1in}- \int_{\SG/\Gamma} \phi\; d\sg \int_\SK \nu\; dm \end{array}\right| > \varepsilon \right\}\right|\le\left|\bigcap^\infty_{j=1} E_{\Ct_j} \right|=0.
\]

Let $A(r)=\left\{ x \in \SG/\Gamma : \alpha^{}_S(x) > r \right\}$. Choose a smooth non-negative function $g_r : \SG/\Gamma \rightarrow [0,1]$ such that
\begin{enumerate}
\item $g_r$ is $\SK$-invariant, i.e., $g_r(\sk x)=g_r(x)$ for all $\sk \in \SK$;
\item $g_r(x)=1$ if $x$ is in $A(r+1)$ and $g_r(x)=0$ if $x$ is outside of $A(r)$. 
\end{enumerate}
 
Take $\phi_1=\phi g_r$ and $\phi_2=\phi (1-g_r)$. Then $\phi=\phi_1 + \phi_2$ and $\phi_2$ is compactly supported. 
Choose $\theta >0$ such that $1 \le s+\theta<n/2$. Then for any $x \in \SG/\Gamma$,
\[\begin{split}
\phi_1(x)&=\phi(x) g_r(x)\le C_\phi \alpha^s_S(x) g_r(x)=C_\phi \alpha^{-\theta}_S(x)\alpha^{s+\theta}_S(x)g_r(x) \\
&\le C_\phi r^{-\theta} \alpha^{s+\theta}_S(x).
\end{split}\]
Put $\psi=C_\phi r^{-\theta} \alpha^{s+\theta}_S(x)$. By the above inequalities and Proposition~\ref{lemma 3.10},
\begin{equation*}
|\phi_1|\le \psi \;\text{and}\; \psi \in \mathcal L^1(\SG/\Gamma)\cap\mathcal L^2(\SG/\Gamma).
\end{equation*}

Let $f\in \mathcal L^2_0(\SG/\Gamma)$ be either a smooth function of compact support or a left $\SK$-invariant function. 
Later, we will put $f=\phi_2-\int_\SK \phi_2$ or $\psi-\int_\SK \psi$.
By Theorem~\ref{matrix coeff}, there are $\lambda_f$, $C_f>0$ such that for any $\sg \in \SG$,
\begin{equation}\label{eq prop 3.13 (5)}
\left|\int_{\SG/\Gamma} f(\sg x)f(x)dx\right|\le C_f \exp\left(-\lambda_f \sum_{p \in S} d_p(g_p, 1)\right),
\end{equation}
where $\sg=(g_p)_{p\in S}$. For instance, we can take $\lambda_f=\min\{\lambda_p, p\in S\}$, where $\lambda_p$'s are in Proposition~\ref{matrix coeff}. Note that
\[ 
\begin{split}
\| A_{\Ct} f\|^2_2 &=\int_{\SG/\Gamma} \left(\int_{\SK} f(\sa_{\Ct}\sk x)\nu(\sk) dm(\sk)\right)^2 dx \\
&= \int_{\SG/\Gamma} \int_{\SK}\int_{\SK} f(\sa_{\Ct}\sk_1 x) f(\sa_{\Ct}\sk_2 x)\nu(\sk_1)\nu(\sk_2)dm(\sk_2)dm(\sk_1) dx \\
&=\int_{\SG/\Gamma} \int_{\SK}\int_{\SK} f(\sa^{}_{\Ct}\sk^{}_1 \sk^{-1}_2 \sa^{-1}_{\Ct} x) f(x)\nu(\sk_1)\nu(\sk_2)dm(\sk_2)dm(\sk_1) dx\\
&=\int_{\SK}\int_{\SG/\Gamma} f(\sa^{}_{\Ct}\sk \sa^{-1}_{\Ct} x) f(x) dx \int_{\SK}\nu(\sk \sk_2)\nu(\sk_2)dm(\sk_2)dm(\sk)\\
&=\int_{\SK}\left(\int_{\SG/\Gamma} f(\sa^{}_{\Ct} \sk \sa^{-1}_{\Ct}x)f(x) dx\right)(\nu*\hat\nu)(\sk) dm(k), 
\end{split}
\]
where $\hat\nu(\sk):=\nu(\sk^{-1})$, $\forall\sk \in \SK$. By Lemma~\ref{lemma 3.12}, 
\[ \left|\mathcal U:=\left\{\sk \in \SK : \sum_{p \in S} d_p(\sa^{}_{\Ct} \sk \sa^{-1}_{\Ct}, 1 ) \le \hab(\Ct) \right\}\right|
\le C_S \left({\hab(\Ct)} \right)^{|S_f|} \exp\left(-{\hab(\Ct)}\right).
\]

Hence by Cauchy-Schwartz inequality, we have that
\begin{equation}\label{eq prop 3.12 (6)} 
\begin{split}
&\left|\int_{\mathcal U}\int_{\SG/\Gamma} f(\sa^{}_{\Ct} \sk \sa^{-1}_{\Ct} x) f(x) dx (\nu *\hat \nu)(\sk) dm(\sk)\right|\\
&\hspace{1.5in}\le  \max(\nu *\hat \nu) \|f\|^2_2\; C_S \left( {\hab(\Ct)}\right)^{|S_f|} \exp\left(- {\hab(\Ct)} \right).
\end{split}
\end{equation}

It is deduced from \eqref{eq prop 3.13 (5)} that
\begin{equation}\label{eq prop 3.12 (7)}
\left|\int_{\SK-\mathcal U} \int_{\SG/\Gamma} f(\sa^{}_{\Ct}\sk\sa^{-1}_{\Ct}x)f(x) dx (\nu*\hat\nu)(\sk) dm(\sk)\right|
\le \max(\nu*\hat\nu)C_f \exp\left(-\lambda_f \hab(\Ct)\right).
\end{equation}

Hence by combining \eqref{eq prop 3.12 (6)} and \eqref{eq prop 3.12 (7)}, there are $\Ct^1_0\succ0$ and a constant $\lambda^1=\lambda^1(f, \Ct^1_0)>0$ such that for $\Ct\succ\Ct^1_0$,
\begin{equation*}\|A_{\Ct} f\|^2_2 \le \exp(-\lambda^1 \hab(\Ct)).
\end{equation*}

\vspace{0.2in}
Now let us show following inequalities:   
\begin{equation}\label{eq prop 3.13 (2)}
\left|\left\{x \in \SG/\Gamma : \left|A_{\Ct'}\phi_1 - \int_{\SG/\Gamma} \phi_1 \int_{\SK} \nu \right|> \frac {\varepsilon} 2 \;\text{for some}\; \Ct'\succ \Ct \right\}\right| < C_1 \exp(-\lambda_{\psi} \hab(\Ct)),
\end{equation}

\begin{equation}\label{eq prop 3.13 (3)}
\left|\left\{x \in \SG/\Gamma : \left|A_{\Ct'}\phi_2 - \int_{\SG/\Gamma} \phi_2 \int_{\SK} \nu \right|> \frac {\varepsilon} 2 \;\text{for some}\; \Ct'\succ \Ct \right\}\right| < C_2 \exp(-\lambda_{\phi_2} \hab(\Ct))
\end{equation}
for some positive constants $C_1$ and $C_2$. 
It is obvious that \eqref{eq prop 3.13 (2)} and \eqref{eq prop 3.13 (3)} imply \eqref{eq prop 3.13}.

First we note that $\alpha^{}_S(\sg x)/\alpha^{}_S(x) \le \max\{\prod_{p\in S}\|\wedge^i(\sg)\|_p : i=1, \ldots, n\}$. Let $\Ct_{\tau}=(\tau, 0, \ldots, 0)$. Then there is $M>0$ such that for all $x\in \SG/\Gamma$, for all $\tau \in [0,1]$, 
$\alpha^{s+\theta}_S(\sa_{\Ct_\tau} x)\le M\alpha_S^{s+\theta}(x)$. 
Hence we have
\[\psi (\sa_{\Ct_\tau} x) \le M \psi(x), \;\forall\tau\in[0,1].\]

Choose $r>0$ sufficiently large so that $\|\phi_1\|_1 \le \|\psi\|_1 \le \varepsilon/(8\max(\nu*\hat\nu)M)$.
Then
\[
\begin{split}
&\left\{x \in \SG/\Gamma : \left|A_{\Ct+\Ct_\tau}\phi_1 - \int \phi_1 \int \nu \right|> \frac \varepsilon 2\;\text{for some}\; \tau \in [0,1] \right\}\\
&\hspace{0.4in}\subseteq\left\{x \in \SG/\Gamma : \left|A_{\Ct+\Ct_\tau}\phi_1 \right|> \frac \varepsilon 4\;\text{for some}\; \tau \in [0,1] \right\}\\
&\hspace{0.4in}\subseteq\left\{x \in \SG/\Gamma : \left|A_{\Ct}\phi_1 \right|> \frac \varepsilon {4M} \; \right\}
\subseteq\left\{x \in \SG/\Gamma : \left|A_{\Ct}\psi \right|> \frac \varepsilon {4M} \; \right\}\\
&\hspace{0.4in}\subseteq\left\{x \in \SG/\Gamma : \left|A_{\Ct}\phi_1 - \int \psi \int \nu \right|> \frac \varepsilon {8M} \; \right\}.
\end{split}
\]

By taking $f=\psi - \int_{\SG/\Gamma} \psi$ in \eqref{eq prop 3.12 (7)} and using the Chebyshev's inequality, there are $\Ct^{\psi}_0\succ0$ and $\lambda_\psi>0$ such that for any $\Ct \succ \Ct^{\psi}_0$,
\[\begin{split}
&\left|E(\Ct, \Ct+\Ct_1)=\left\{x \in \SG/\Gamma : \left|A_{\Ct+\Ct_\tau}\phi_1 - \int \phi_1 \int \nu \right|> \frac \varepsilon 2\;\text{for some}\; \tau \in [0,1] \right\}\right|\\
&\hspace{0.8in} \le \left(\frac {8M} {\varepsilon} \right)^2 \exp (-\lambda_{\psi} \hab(\Ct)).
\end{split}\]

Hence using the geometric series argument, there is a constant $C_1>0$ such that the inequality~\eqref{eq prop 3.13 (2)} holds.
\[
\begin{split}
&\left|\left\{x \in \SG/\Gamma : \left|A_{\Ct'}\phi_1 - \int \phi_1 \int \nu \right|> \frac \varepsilon 2\;\text{for some}\; \Ct'\succ \Ct \right\}\right|\\
&\hspace{1in}\le \sum_{\scriptsize \begin{array}{cc}t'_\infty=t_\infty+n,\\
\forall n\in \NN\cup\{0\}\end{array}}\sum_{\scriptsize\begin{array}{cc}\forall t'_p\ge t_p,\\ p\in S_f\end{array}} \left|E(\Ct', \Ct'+\Ct_1) \right|
\le\; C_1 \exp(-\lambda_{\psi} \hab(\Ct)).
\end{split}
\]

For \eqref{eq prop 3.13 (3)}, note that since $\phi_2$ is compactly supported, $\phi_2$ is uniformly continuous. Hence there is $\delta>0$ such that
\[ \left|A_{\Ct+\Ct_\tau} \phi_2(x) - A_{\Ct} \phi_2(x)\right|<\frac {\varepsilon} 4 \]
for all $\Ct\succ 0$, $x\in \SG/\Gamma$ and $\tau\in[0,\delta]$.
Again by \eqref{eq prop 3.12 (7)} and Chebyshev's inequality, there are $\Ct^{\phi_2}_0\succ 0$ and $\lambda_{\phi_2}>0$ such that for all $\Ct \succ \Ct^{\phi_2}_0$,
\[
\begin{split}
&\left|\left\{x \in \SG/\Gamma : \left|A_{\Ct+\Ct_\tau}\phi_2 - \int_{\SG/\Gamma} \phi_2 \int_{\SK} \nu \right|> \frac {\varepsilon} 2 \;\text{for some}\; \tau \in [0, \delta] \right\}\right|\\
&\hspace{0.4in}\le \left|\left\{x \in \SG/\Gamma : \left|A_{\Ct}\phi_2 - \int_{\SG/\Gamma} \phi_2 \int_{\SK} \nu \right|> \frac {\varepsilon} 4 \; \right\}\right| \le \left(\frac 4 \varepsilon \right)^2 \exp(-\lambda_{\phi_2} \hab(\Ct)).
\end{split}
\]
Then \eqref{eq prop 3.13 (3)} follows from the similar geometric argument used above.
\end{proof}

From now on, a function $f_p$ on $\QQ_p^n$, $p \in S$, is always assumed to be compactly supported.  
If $p < \infty$, we additionally assume that $f_p$ is $(\ZZ_p-p\ZZ_p)$-invariant:
\begin{equation}\label{UFCf}
f_p(ux_1, x_2, \ldots, x_{n-1}, u^{-1}x_n)=f_p(x_1, x_2, \ldots, x_n),\;\forall u \in \ZZ_p -p\ZZ_p.
\end{equation}

Let $\nu$ be the product of non-negative continuous functions $\nu_p$ 
on the unit sphere in $\QQ_p^n$, $p \in S$. For $p\in S_f$, we also assume that
\begin{equation}\label{UFCnu}
\nu_p(u\ov)=\nu_p(\ov), \;\forall u \in \ZZ_p - p\ZZ_p.
\end{equation}

Define a function $J_f$ for $f=\prod_{p\in S}f_p$ by $J_f=\prod_{p\in S} J_{f_p}$, where
\[\begin{split}
J_{f_\infty}(r, \zeta_\infty)&= \frac 1 {r^{n-2}} \int_{\RR^{n-2}} f_\infty(r, x_2, \ldots, x_{n-1}, x_n) dx_2 \cdots dx_{n-1},\\
J_{f_p}(p^{-r},\zeta_p)&=\frac 1 {p^{r(n-2)}}\int_{\QQ_p^{n-2}} f_p(p^{-r},
x_2, \ldots, x_{n-1}, x_n) \ dx_2\cdots
dx_{n-1},\;p\in S_f,
\end{split}\]
where $x_n$ is determined by the equation $\zeta_p=q^0(p^{-r}, x_2, \ldots,x_{n-1}, x_n)$ (If $p=\infty$, replace $p^{-r}$ by $r$). 
By Lemma 3.6 in \cite{EMM} and Lemma 4.1 in \cite{HLM}, for sufficiently small $\varepsilon>0$, there are $c(K_p)>0$ and $t_p^0> 0$ for each $p \in S$, such that if $t_p>t^0_p$,
\begin{equation}\label{real Jf}
\begin{split}
&\left|c(K_\infty) e^{t_\infty(n-2)}\int_{K_\infty} \hspace{-0.1in}f_\infty(a_{t_\infty}k_\infty\ov)\nu(k^{-1}_\infty\ve_1)dm(k_\infty)\right.\\
&\hspace{1.7in} \left.-J_{f_\infty}(\|\ov\|_\infty e^{-t_\infty}, q^0_\infty(\ov))\nu(\frac{\ov}{\|\ov\|^\sigma_\infty})\right|<\varepsilon,
\end{split}\end{equation}
\begin{equation}\label{p-adic Jf}\begin{split}
&\left|c(K_p)p^{t_p(n-2)} \int_{K_p} f_p(a_{t_p}k_p\ov)\nu_p(k^{-1}_p \ve_1)dm(k_p)\right.\\
&\hspace{1.4in}-\left.J_{f_p}(p^{t_p}\|\ov\|_p^{\sigma}, q^0_p(\ov))\nu(\frac{\ov}{\|\ov\|_p^{\sigma}})\right|<\varepsilon,\; p \in S_f.
\end{split}\end{equation}

Recall that $\sigma=1$ for the infinite place and $\sigma=-1$ for the finite place.
Define
$$\|g\ov\|^\sigma / \CT^\sigma=(\|g_p\ov\|^\sigma_p/T^\sigma_p)_{p\in S}.$$

By Lemma 3.6 in \cite{EMM} and Lemma 4.1 in \cite{HLM}, there is a constant $c(\SK)>0$ such that for sufficiently small $\varepsilon>0$ and sufficiently large $\Ct \succ 0$, 
\begin{equation}\label{eq prop 3.6}
\begin{split}
&\left| J_f \left( \frac{\| \sg \ov \|^\sigma}{\T^\sigma}, \q( \sg \ov) \right) \nu \left( \frac{\sg\ov}{\| \sg \ov \|^\sigma} \right)\right.\\
&\hspace{1.5in}\left.-c(\SK)|\CT|^{n-2}\int_\SK
\tilde{f}(a_\t \sk \sg)\nu(\sk^{-1}\ve_1)dm(\sk)\right|\leq \varepsilon. \end{split}
\end{equation}

\begin{corollary}\label{upper bound} 
Let $(\CT_j=(T_{j,p})_{p\in S})_{j\in\NN}$ be a divergent sequence. Define
$$S'=\{p \in S : (T_{j,p})_{j\in \NN}\;\text{is bounded}\;\}\subsetneq S.$$

Let $\Omega, \I_S=(I_p)_{p\in S}$ be as in Theorem \ref{main thm}.
Then for almost all nondegenerate isotropic quadratic form $\q=(q_p)_{p\in S}$, 
there is a constant $C=C(S')>0$ such that
\begin{equation*}
\left|\left\{\ov \in \ZZ_S^n \cap \CT_j\Omega : \q(\ov)\in \I_S\right\}\right|< C \prod_{p\in S-S'}(T_{j,p})^{n-2},
\end{equation*}
where $n$ is the rank of $\q$.
\end{corollary}
\begin{proof}
Since we want to show the upper bound, for simplicity,
we may assume that $\Omega$ is the product of unit balls in $\QQ_p^n$, $p \in S$.

Let $\varepsilon>0$ be given.
For each $p \in S$, let $\mathcal C_p$ be a compact set of the space of nondegenerate isotropic quadratic forms over $\QQ_p$ of a given signature.
Let $\mathcal C=\prod_{p \in S} \mathcal C_p$.

Let $\sg_{\q}$ be an element of $\SL_n(\QQ_S)$ such that $\q(\ov)=\sq(\sg_{\q}\ov)$ for all $\ov \in \QQ_S^n$.
Then for each $p \in S$, there is $\beta_p=\beta_p(\mathcal C_p)>0$ such that if $\q \in \mathcal C$, $\beta_p^{-1}\le \|\sg_{\q} \ov\|_p/\|\ov\|^\sigma_p \le \beta_p$ for all $\ov \in \QQ_S^n$.

Choose bounded continuous functions $f_p$, $p \in S$, of compact support such that 
\[J_{f_p}\ge 1+\varepsilon\quad\text{on}\quad [\frac 1 {2\beta_p}, \beta_p]\times I_p.
\]

If $\ov$ satisfies that $T_p/2 \le \|\ov\|_p \le T_p$ and $q_p(\ov)\in I_p$, then 
$$J_{f_p}(\|\sg\ov\|_p/T^{\sigma}_p, q^0_p(\sg\ov))\ge 1+\varepsilon$$ for $\sg=\sg_{\q}$ with $\q \in \mathcal C$.
By \eqref{real Jf} and \eqref{p-adic Jf}, there is $\CT^0=(T^0_p)_{p\in S}\succ 0$  
such that for each $p \in S$ and for all $T_p > T^0_p$,

\[\left|c(K_p)T_p^{n-2}\int_{K_p} f_p(a_{t_p}k_p \ov) dm(k_p)
- J_{f_p}\left(\frac{\|\ov\|_p^{\sigma}}{T^\sigma_p}, \zeta\right)\right|<\varepsilon.
\]

For $p \in S'$,
we may further assume that $T_{j,p}\le T^0_p$ for all $j$.

Finally for each $p \in S$, choose a nonnegative bounded function $g_p$ on $\QQ_p^n$ of compact support such that
if $\|\ov\|_p\le T^0_p$ and $t_p\le \log_p(T_p^0)$,
\[
\int_{K_p} g_p(a_{t_p} k_p \ov) dm(k_p) \ge 1.
\]

Take $h_{S'}=\prod_{p \in S'} g_p\times\prod_{p \in S-S'} f_p$ and
let $\widetilde{h_{S'}}$ be the Siegel transform of $h_{S'}$ defined in Lemma \ref{Schmidt lemma}.
Then for $\CT=(T_p)$ such that $T_p>T_p^0$, $\forall p \in S-S'$, we obtain that
\[\begin{split}
&\left|\left\{\ov \in \ZZ_S^n :\begin{array}{c}
\|\ov\|_p\le T_p^0, \;\forall p \in S'\\
T_p/2<\|\ov\|_p\le T_p,\;\forall p \in S-S'
\end{array}, \;\q(\ov)\in \I_S \right\}\right|\\
&\le \sum_{\ov \in \ZZ_S^n}\left( \prod_{p \in S'}\int_{K_p} g_p(a_{t_p}k_p\sg_{\q} \ov)dm(k_p)\times\hspace{-0.15in}\prod_{p\in S-S'}\hspace{-0.08in}T_p^{n-2}\int_{K_p} f_p(a_{t_p}k_p\sg_{\q}\ov)dm(k_p)\right)\\
&=\left(\prod_{p\in S-S'} T_p^{n-2}\right)\int_{\SK} \widetilde {h_{S'}} (\sa_{\Ct}\sk\sg_{\q} \ov)dm(\sk).
\end{split}\]

By Proposition \ref{prop 3.13}, for almost all quadratic form $\q$, as $\Ct$ diverges,
\[\int_{\SK} \widetilde{h_{S'}}(\sa_{\Ct}\sk \sg_{\q}\SL_n(\ZZ_S))dm(\sk)\rightarrow \int_{\SG/\Gamma} \widetilde{h_{S'}} d\sg <\infty.
\]

Hence there is a constant $C'>0$ such that for all $\CT_j$,
\[\left|\left\{\ov \in \ZZ_S^n :\hspace{-0.08in}\begin{array}{c}
\|\ov\|_p\le T_p^0, \;\forall p \in S'\\
T_{j,p}/2<\|\ov\|_p\le T_{j,p},\;\forall p \in S-S'
\end{array}\hspace{-0.05in}, \;\q(\ov)\in \I_S \right\}\right|< C'\hspace{-0.1in}\prod_{p\in S-S'} (T_{j,p})^{n-2}.
\]
Therefore we have 
\[\left|\ov \in \ZZ_S^n \cap \CT_j\Omega : \q(\ov)\in \I_S\}\right|
< \left( C'\hspace{-0.05in}\prod_{p \in S-S'}\frac {p}{p-1}\right)
\left(\prod_{p\in S-S'} T_{j,p}\right)^{n-2}. 
\]
\end{proof}

\section{The proof of the main theorem}

The proof of Theorem~\ref{main thm} is similar to that of Theorem 1.5 in \cite{HLM}
 except we use Proposition~\ref{prop 3.13} instead of Theorem 7.1 in \cite{HLM}. 
Let $\mathcal C$ be any compact subset of the space of isotropic quadratic forms with equal signature and let $\varepsilon>0$ be given. 
Let $\phi$ and $\nu$ be compactly supported functions defined as before.

Theorem 7.1 in \cite{HLM} says that there is $\Ct_0=\Ct_0(\mathcal C)\succ0$ such that
except on a finite union of orbits of $\SO(\sq)$, $x \in \mathcal C$ satisfies the following:
for all $\Ct \succ \Ct_0$,
 \[\left|\int_{\SK} \phi(\sa_{\Ct}\sk x)\nu(\sk)dm(\sk)-
 \int_{\SG/\Gamma} \phi d\sg \int_{\SK} \nu dm \right|<\varepsilon.
 \]

\begin{proposition}\label{prop 3.7} Let $\q$ be an isotropic quadratic form of rank $n=3$ or $4$. 
Let $f=\prod f_p$ and $\nu=\prod \nu_p$ be as before.
Assume further that $f$ satisfies the following condition:
there is a nonnegative continuous function $f^{+}$ of compact support on $\QQ_S^n$ such that $\supp(f)\subset \supp(f^+)^\circ$, where $A^\circ$ is an interior of $A$ and
\begin{equation}\label{eq prop 3.7 (3)}
\sup_{\Ct \succ 0} \int_{\SK} \widetilde{f^+} (\sa_{\Ct} \sk \sg\Gamma) dm(\sk) = M < \infty.
\end{equation}

Then there exists $\t_0\succ 0$ such that if $\t \succ \t_0$,
\begin{equation}\label{eq prop 3.7}
\begin{split}
&\left| |\T|^{-(n-2)}\sum_{\ov\in\ZZ_S^n} 
J_f \left( \frac{\| \sg \ov \|^\sigma}{\T^\sigma}, \q( \sg \ov) \right) \nu \left( \frac{\sg\ov}{\| \sg \ov \|^\sigma} \right)\right.\\
&\hspace{2in}\left.-c(\SK)\int_\SK
\tilde{f}(a_\t \sk \sg)\nu(\sk^{-1}\ve_1)dm(\sk)\right|\leq \epsilon. 
\end{split}
\end{equation}
\end{proposition}
\begin{proof} We first claim that there is a constant $c>0$ such that 
\begin{equation}\label{eq prop 3.7 (2)}
\left|\Pi:=\left\{\ov \in \ZZ^n_S : J_f \left( \frac{\| \sg \ov \|^\sigma}{\T^\sigma}, \q( \sg \ov) \right) \nu \left( \frac{\sg\ov}{\| \sg \ov \|^\sigma} \right)\neq 0\right\}\right|< c |\CT|^{n-2}.
\end{equation}

Since the interior of $\supp(f^+)$ contains $\supp(f)$, there is $\rho>0$ such that
\[ J_{f^+} > \rho \quad \text{on}\; \supp (J_f).
\]

Note that since $J_{f^+}$ is compactly supported, the set $\Pi$ is finite.
Take $\varepsilon>0$ such that $\varepsilon < \rho/2$. Then by \eqref{eq prop 3.6}, for sufficiently large $\Ct$, we have that
\[\rho|\Pi|\le\sum_{\ov \in \Pi} J_{f^+}\left( \frac{\| \sg \ov \|^\sigma}{\T^\sigma}, \q( \sg \ov) \right)
\le c(\SK) |\CT|^{n-2}\int_{\SK} \widetilde{f^+}(\sa_\Ct\sk\sg \Gamma)dm(\sk)+\varepsilon|\Pi|.
\]

By \eqref{eq prop 3.7 (3)}, $\rho/2\; |\Pi| \le c(\SK)M|\CT|^{n-2}$. This implies \eqref{eq prop 3.7 (2)}.

Now \eqref{eq prop 3.7} follows from applying \eqref{eq prop 3.6} by putting $\varepsilon/(c(\SK)\rho M)$ instead of $\varepsilon$ and taking summation over all $\ov \in \ZZ_S^n$.
\end{proof}

Recall that a convex set $\Omega$ is defined using a non-negative continuous function $\rho=\prod_{p\in S} \rho_p$, 
where $\rho_p$ is a positive function on the unit sphere of $\QQ_p^n$, $p \in S$.

Define the shell $\hat\Omega$ of $\Omega$ by
\[\hat\Omega= \left\{\ov \in \QQ_S^n : \rho_p(\ov/\|\ov\|_p^\sigma)/2 < \|\ov\|_p \le \rho_p(\ov/\|\ov\|_p^\sigma), \;\forall p \in S \right\}. 
\]

Note that when $p<\infty$, 
the inequality in the above definition is in fact the equality: $\|\ov\|_p=\rho_p(\ov/\|\ov\|^\sigma_p)$. 

The following proposition was originally stated for $\Omega$ but the proof can be easily modified for $\hat \Omega$.

\begin{proposition}\cite[Proposition 1.2]{HLM}\label{volume-asym}
 There is a constant $\lambda=\lambda(\q,\Omega)>0$ such that as $\CT \rightarrow \infty$,
\[\vol\left\{\ov \in \QQ_S^n\cap \CT\hat\Omega : \q(\ov)\in \I_S \right\}\sim \lambda(\q,\hat\Omega) \cdot | \I_S | \cdot |\CT|^{n-2}.
\]
\end{proposition}

\begin{lemma}\cite[Lemma 5.2]{HLM}
Let $f$ and $\nu$ be as in Proposition \ref{prop 3.7}.
Take
\[
h_p( \ov_p, \zeta_p)=J_{f_p}(\|\ov_p\|_p^{\sigma}, \zeta_p)\nu_p(\ov_p/\|\ov_p\|_p^{\sigma}),\quad p \in S
\]
and set $h(\ov, \zeta)= \prod_{p \in S}  h_p( \ov_p, \zeta_p)$. 
Then we have
\begin{equation}\label{eq lemma 3.9}
\begin{split}
&\lim_{\CT \rightarrow \infty} |\CT|^{-(n-2)} \int_{\QQ_S^n} h\left(\frac \ov {\CT^\sigma}, \q(\ov)\right) d\ov\\
&\hspace{2in}= c(\SK)
 \int_{\SG/\Gamma}\tilde{f}(\sg)\, d\sg  \prod_{p \in S} \int_{K_p}
 \nu_p(k_p^{-1}\ve_1)dm(k_p).
\end{split}\end{equation}
\end{lemma}

\vspace{0.2in}
\begin{proof}[Proof of Theorem~\ref{main thm}]
Since $f$ in Proposition \ref{prop 3.7} is compactly supported,
there is $f^+$ such that $\supp(f_p)\subset \supp(f^+_p)^\circ$. 
By Proposition~\ref{prop 3.13} with $\phi=f^+$ and $\nu\equiv1$,
\[\lim_{\Ct\rightarrow \infty}\int_{\SK} \widetilde{f^+}(\sa_\Ct \sk \sg \Gamma)dm(\sk)=\int_{\SG/\Gamma}\widetilde{f^+} d\sg
\]
for almost all $x=\sg\Gamma \in \SG/\Gamma$.
By Lemma \ref{Schmidt lemma} and Proposition \ref{lemma 3.10}, the integral of $\widetilde{f^+}$ over $\SG/\Gamma$ is finite.
Hence one can apply Proposition \ref{prop 3.7} for almost all $x \in \SG/\Gamma$.

We also remark that the set of functions of the form $J_f \nu$, where $f=\prod f_p$, $\nu=\prod \nu_p$ with \eqref{UFCf} and \eqref{UFCnu} respectively, is a generating set of 
$$\mathcal L=\{F(\ov, \zeta) : \left(\QQ_S^n\right) \times \QQ_S\rightarrow \RR\; | \;F(u\ov, \zeta)=F(\ov, \zeta),\; \forall u \in \ZZ_p - p\ZZ_p,\; \forall p \in S_f\}.$$

Hence Proposition~\ref{prop 3.7} holds for functions in $\mathcal L$ as well (see details in \cite{HLM}). Define
\[L(h):=\lim_{\CT \rightarrow \infty} |\CT|^{-(n-2)} \int_{\QQ_S^n} h\left(\frac \ov {\CT^\sigma}, \q(\ov)\right) d\ov,\; h \in \mathcal L.
\]

The characteristic function $\mathbbm 1_{\CT\hat\Omega\times\I_S}(\sg\ov,\zeta)$ is contained in $\mathcal L$.
Let $\varepsilon>0$ be given. 
Take continuous functions $h^a, h^b \in \mathcal L$, depending on $\CT$, $\varepsilon$ and $\q$ such that
\begin{equation}\label{eq last (3)}
h^b(\sg \ov, \zeta) \le \mathbbm 1_{\CT\hat\Omega\times\I_S}(\sg\ov,\zeta) \le h^a(\sg \ov, \zeta)\quad\text{and}\quad
\left|L( h^a) -L( h^b) \right| < \varepsilon.
\end{equation}

From \eqref{eq prop 3.7}, \eqref{eq prop 3.13} and \eqref{eq lemma 3.9}, 
for $h=h^a$, $h^b$ and any $\varepsilon>0$, there is $\CT_0 \succ 0$ such that if $\CT \succ \CT_0$,
\begin{equation}\label{eq last (1)}
\left| |\CT|^{-(n-2)} \sum_{\ov \in \ZZ^n_S} h\left( \frac {\sg \ov}{\CT^\sigma}, \q(\sg \ov)\right) - L(h) \right| < \varepsilon
\end{equation}
for almost every $\sg\Gamma\in \SG/\Gamma$.
By the definition of $L(h)$ and rescaling $\CT_0 \succ 0$ if necessary, we also obtain that
\begin{equation}\label{eq last (2)}
\left| |\CT|^{-(n-2)} \int_{\QQ_S^n}h\left( \frac {\sg \ov}{\CT^\sigma}, \q(\sg \ov)\right)  -L(h)\right|<\varepsilon.
\end{equation}

Combining \eqref{eq last (1)} and \eqref{eq last (2)} with \eqref{eq last (3)}, if we regard $h$ as the characteristic function of $\CT\hat\Omega\times \I_S$,
\begin{equation*}
\left|\;\left|\left\{\ov \in \ZZ_S^n \cap \CT\hat\Omega : \q(\ov) \in \I_S  \right\}\right| - \vol\left(\left\{\ov \in \QQ_S^n \cap \CT\hat\Omega : \q(\ov) \in \I_S \right\}\right)\right|<4\varepsilon.
\end{equation*}

Since
\[\begin{split}&\left|\left\{\ov \in \ZZ_S^n \cap \CT \Omega : \q(\ov) \in \I_S \right\}\right|\\
&\hspace{0.4in}=
\hspace{-0.2in}\sum_{\scriptsize \begin{array}{c} n_j\in \{0\}\cup\NN\\j\in\{0,\ldots, s\}\end{array}}\hspace{-0.1in}
\left|\left\{\ov \in \ZZ_S^n \cap (2^{-n_0}T_\infty,  p^{-n_1}_1 T_1, \ldots, p^{-n_s}_s T_s)\hat \Omega : \q(\ov) \in \I_S \right\}\right|,
\end{split}\]
the lemma follows from Proposition~\ref{volume-asym} and the classical argument of geometric series, if we obtain the following: as $\CT \rightarrow \infty$, 
the summation over all $\CS=(S_\infty, S_1, \ldots, S_s)=(2^{-n_0}T_\infty,  p^{-n_1}_1 T_1, \ldots, p^{-n_s}_s T_s)$ with $\CS \nsucc \CT_0=(T^0_\infty, T^0_1, \ldots, T^0_s)$ is 
\[\sum_{\CS}\left|\left\{\ov \in \ZZ_S^n \cap \CS\hat\Omega : \q(\ov) \in \I_S  \right\}\right|=o(|\CT|^{n-2})\quad\text{as}\quad \CT\rightarrow \infty.
\]

For this, by rescaling $\CT^0$ if necessary, let us assume that Corollary \ref{upper bound} holds for any $S'\subsetneq S$. 
Denote the set $\{\CS=(2^{-n_0}T_\infty,  p^{-n_1}_1 T_1, \ldots, p^{-n_s}_s T_s) \prec \CT : \CS \nsucc \CT_0\}$ by $\cup\{\Psi_{S'}: S'\subseteq S\}$, where
\[\Psi_{S'}:=\left\{\CS=(S_\infty, S_1, \ldots, S_s) : 
S_p \le T^0_p,\;\forall p\in S'\;\text{and}\;
S_p > T^0_p,\;\forall p\in S-S'\right\}.
\]

Then for $S'\subsetneq S$, by Corollary \ref{upper bound}, there is a constant $C(S')>0$ such that
\begin{equation}\label{eq last (5)}\begin{split}
\sum_{\CS \in \Psi_{S'}}\left|\left\{\ov \in \ZZ_S^n \cap \CS\hat\Omega : \q(\ov) \in \I_S  \right\}\right|&<
\left|\{\ov \in \ZZ_S^n\cap \CT_{S'}\Omega : \q(\ov)\in \I_S\}\right|\\
&<C(S')\prod_{p\in S-S'} (T_p)^{n-2},
\end{split}\end{equation}
where $\CT_{S'}=(T'_p)_{p\in S}$. Here $T'_p=T^0_p$ for $p\in S'$ and $T'_p=T_p$ for $p\in S-S'$. Hence the left hand side of \eqref{eq last (5)} is $o(\prod_{p \in S-S'} T_p^{n-2})$.
If $S'=S$, since
\[\sum_{\CS \in \Psi_{S}}\left|\left\{\ov \in \ZZ_S^n \cap \CS\hat\Omega : \q(\ov) \in \I_S  \right\}\right|\le \left|\left\{\ov \in \ZZ_S^n \cap \CT_0\Omega : \q(\ov) \in \I_S  \right\}\right|<\infty,
\]
it is obviously $o(|\CT|^{n-2})$.
\end{proof}

\end{document}